\documentclass[11pt]{article}

\usepackage{caption}
\usepackage{subcaption}
\usepackage{geometry}                
\usepackage[parfill]{parskip}    
\usepackage{graphicx}
\usepackage{amssymb}
\usepackage{epstopdf}
\usepackage{amsmath,bbm}
\usepackage{tabularx}
\usepackage{algorithmicx}
\usepackage[ruled]{algorithm}
\usepackage{algpseudocode,enumerate}
\usepackage[numbers]{natbib}
\usepackage{amsthm}
\usepackage{todonotes}
\usepackage{color,soul}

\providecommand{\keywords}[1]{\textbf{{Keywords }} #1}
\providecommand{\subjclass}[1]{\textbf{{Mathematics subject classification }} #1}

\newtheorem{mytheorem}{Theorem}

\newtheorem{remark}{Remark}
\newtheorem{corollary}{Corollary}

\newcommand{\vect}[1]{\boldsymbol{#1}}

\newcommand{\loglog}{\mathop{\rm loglog}}

\numberwithin{equation}{section}

\algrenewcommand\algorithmicindent{1.0em}%

\allowdisplaybreaks

\title{Multiscale High-Dimensional Sparse Fourier Algorithms \\for Noisy Data}
\author{Bosu Choi}

\author{%
Bosu Choi\thanks{The Oden Institute for Computational Engineering and Sciences, University of Texas at Austin \texttt{choibosu@utexas.edu}}
\and
Andrew Christlieb  \thanks{Department of Computational Mathematics, Science, and Engineering
(CMSE), Michigan State University, \texttt{christli@msu.edu}}%
\and
Yang Wang\thanks{Department of Mathematics, The Hong Kong University of Science and Technology, \texttt{yangwang@ust.hk‎} }
}

\date{}

\begin{document}

\maketitle

\begin{abstract}
We develop an efficient and robust high-dimensional sparse Fourier algorithm for noisy samples. Earlier in the paper {\em Multi-dimensional sublinear sparse Fourier algorithm} (2016) \cite{2016arXiv160607407C}, an efficient sparse Fourier algorithm with $\Theta(ds \log s)$ average-case runtime and $\Theta(ds)$ sampling complexity under certain assumptions was developed for signals that are $s$-sparse and bandlimited in the $d$-dimensional Fourier domain, i.e. there are at most $s$ energetic frequencies and they are in $ \left[-N/2, N/2\right)^d\cap \mathbb{Z}^d$. However, in practice the measurements of signals often contain noise, and in some cases may only be nearly sparse in the sense that they are well approximated by the best $s$ Fourier modes. In this paper,  we propose a multiscale sparse Fourier algorithm for noisy samples that proves to be both robust against noise and efficient. 
\end{abstract}



\section{Introduction} \label{sec:intro}
\setcounter{equation}{0}

Sparsity or compressibility in large data sets appears in many applications. Efficient algorithms taking advantage of these properties have been developed in different areas \cite{allard2012multi, candes2005decoding, coifman2006diffusion, iwen2017robust} in order to reduce sampling and/or runtime complexities. Compressive sensing \cite{donoho2006compressed, foucart2013mathematical} demonstrates that an appropriate small number of samples are sufficient to solve under-determined linear systems under certain well known conditions. More precisely, assuming $\vect{x}\in \mathbb{C}^L$ is $s$-sparse then $A \vect{x}=\vect{y}$ where $A\in \mathbb{C}^{I\times L}$,  $\vect{y}\in \mathbb{C}^I$ and $I<L$ can be solved under appropriate conditions on $A$ and $I$. Compressive sensing has led to an explosion in the study of sparsity and algorithms that can take advantage of sparsity to drastically reduce both runtime and sample complexities. One example of the algorithms in this type is the sparse Fourier transform, which finds the energetic Fourier modes for a signal that is $s$-sparse in the Fourier domain quickly with reduced number of samples. %

Several algorithms with varying approaches have been proposed to achieve {\em sublinear} runtime and sampling complexity for sparse Fourier transform, for both one-dimensional and higher-dimensional settings.  
In the one-dimension, the first sparse Fourier transform was a randomized algorithm introduced in \cite{gilbert2002near} having $\mathcal{O}(s^2\log^c N)$ runtime and $\mathcal{O}(s^2\log^c N)$ sampling complexity with a small positive constant $c$. The constant $c$ controls the balance between accuracy and efficiency. In the follow up work \cite{gilbert2005improved} both runtime and sampling complexity were improved to $\mathcal{O}(s\log^c N)$. The randomized algorithms in \cite{hassanieh2012nearly, hassanieh2012simple} have $\mathcal{O}(s\log N\log N/s )$ average-case runtime complexity. The first deterministic algorithm is introduced in \cite{iwen2010combinatorial}, which uses techniques from number theory and combinatorics and has $\mathcal{O}(s^2\log^4N)$ runtime and sampling complexity. In \cite{iwen2013improved}, an improved deterministic algorithm was introduced, and the extension to higher dimensional function was suggested but it suffered the exponential dependence of runtime complexity on the dimension. Another deterministic algorithm was introduced in \cite{lawlor2013adaptive} which uses the similar idea in the frequency recovery through phaseshift and works for noiseless samples from exactly $s$-sparse functions. It has $\Theta(s\log s)$ average-case runtime and $\Theta(s)$ sampling complexity. In \cite{christlieb2016multiscale} the method from \cite{lawlor2013adaptive} was extended to incorporate a multiscale technique that works robustly with noisy samples and has $\Theta(s\log s\log N/s)$ average-case runtime and $\Theta(s\log N/s)$ sampling complexity. 

Extension of one-dimensional spare Fourier transform to the multi-dimensional problem setting is in general not  straightforward. One simple way of extension is to unwrap the multi-dimensional signal to one-dimensional signal. However, this method suffers from the exponentially large runtime complexity due to the curse of dimensionality \cite{iwen2013improved}. The first randomized algorithm for two-dimensional problem was introduced in \cite{ghazi2013sample} through the use of parallel projections of frequencies. It has $\mathcal{O}(s)$ sampling and $\mathcal{O}(s\log s)$ runtime complexity on average for the exactly sparse signals, and $\mathcal{O}(s\log N)$ sampling and $\mathcal{O}(s\log^2 N)$ runtime complexity on average for the approximately sparse signals. In \cite{potts2016sparse} a general $d$-dimensional sparse Fourier algorithm was developed to achieve $\mathcal{O}(ds^2N)$ samples and $\mathcal{O}(ds^3+ds^2N\log(sN))$ runtime complexity. The algorithm uses rank-1 lattices and it finds energetic frequencies in a dimension-incremental fashion. While it is a deterministic algorithm it can also be modified into a randomized algorithm with $\mathcal{O}(ds+dN)$ sample complexity and $\mathcal{O}(ds^3)$ runtime complexity. A randomized algorithm introduced in \cite{kapralov2016sparse} requires $2^{\mathcal{O}(d^2)}(s\log N^d \loglog N^d)$ samples and $2^{\mathcal{O}(d^2)}s\log^{d+3}N^d$ runtime. In \cite{morotti2016explicit}, two deterministic sampling sets $\mathcal{O}(s^2d^2N)$ and $\mathcal{O}((s^2d^3N\log N)$ are constructed and the corresponding algorithms have $\mathcal{O}(s^2d^2N^d)$ and $\mathcal{O}(s^2d^3N^2\log N)$ runtime complexities respectively under the assumption that $N$ is a prime number. In \cite{2016arXiv160607407C} we have developed an algorithm combining phaseshift from \cite{iwen2013improved} and various transformations including parallel projection was given with average-case runtime complexity of $\mathcal{O}(sd\log s)$ and sampling complexity $\mathcal{O}(sd)$ under certain assumptions. 

In this paper, we introduce a multi-dimensional sparse Fourier algorithm for noisy samples. Our algorithm uses the techniques from our algorithm for noiseless samples in \cite{2016arXiv160607407C} and the  multiscale technique from \cite{christlieb2016multiscale} to overcome some of the challenges. Let $f:[0,1)^d\rightarrow \mathbb{C}$ be defined by $f(\vect{x}) =\sum_{j=1}^{s} a_j e^{2\pi i \vect{w}_j \cdot \vect{x}} + n(\vect{x})$ where $\vect{w}_j \in [-N/2,N/2)^d \cap \mathbb{Z}^d$, $s\ll N^d$ and $n(\vect{x})$ represents noise. The algorithms from \cite{2016arXiv160607407C} are not robust to noise since  it needs to compute the ratio of discrete Fourier transforms of samples from $f$ at shifted and unshifted points, which is not robust when the shift is small or the noise level $\sigma$ is high. To overcome this we adopt a multiscale approach similar to the one introduced in \cite{christlieb2016multiscale}. This allows us to progressively approximate the significant modes while controlling the influence of noise. In higher dimensions there are some additional challenges we are able to overcome. Details will be discussed later in the paper. 

Our algorithm assumes that we have access to an underlying continuous function $f$, i.e., we can sample at anywhere we want. However, samples are sometimes given at the beginning and getting extra samples can be very expensive. Accordingly, it is necessary to do approximation of extra samples using given samples. A fully discrete sparse Fourier transform was introduced in \cite{merhi2017new} combining periodized Gaussian filters and one-dimensional sparse Fourier transform in continuous setting such as \cite{iwen2013improved} and \cite{christlieb2016multiscale}. For future work we are hopeful that similar approaches can help us to adopt the multiscale high-dimensional algorithm to the case of fully discrete samples.  

The rest of this paper is organized as follows. In Section \ref{sec:pre} we introduce our problem setting, necessary notation and our noise model, and review briefly about high-dimensional sparse Fourier algorithm from \cite{2016arXiv160607407C}. Section \ref{sec:multiscale} introduces a multiscale method for the high-dimensional sparse Fourier algorithm from our previous work. This modifies the algorithm to be able to recover noisy signals.   In Section \ref{sec:alg}, parameters that determine the performance of our algorithm are introduced and the pseudocode is given with the description and analysis. The results of the numerical experiments are shown in Section \ref{sec:eval} and the conclusion is in Section \ref{sec:con}.

\section{Preliminaries} \label{sec:pre}
\setcounter{equation}{0}

\subsection{Notation and Review} \label{sec:notation} 

In this section, we introduce the notation used throughout the rest of this paper and the brief review of high-dimensional sparse Fourier algorithm for samples without noise in \cite{2016arXiv160607407C}. Let $s$, $d$ and $N$ be natural numbers, $s\ll N^d$, and $D:=[0,1)^d$. We consider a function $f:D\rightarrow \mathbb{C}$ which is $s$-sparse in the $d$-dimensional Fourier domain as follows
\begin{equation*}
f(\vect{x})=\sum_{j=1}^{s} a_j e^{2\pi i \vect{w}_j \cdot \vect{x}}
\end{equation*}
where each $\vect{w}_j \in [-N/2, N/2)^d \cap \mathbb{Z}^d$ and $a_j\in \mathbb{C}$. We note that $f$ can be regarded as a periodic function defined on $\mathbb{R}^d$. The aim of sparse Fourier algorithms is to rapidly reconstruct a function $f$ using small number of its samples. In \cite{2016arXiv160607407C}, the methods were introduced using several different transformations and parallel projections along coordinate axes of frequencies in order to exploit the one-dimensional sparse Fourier algorithm from \cite{lawlor2013adaptive}. These transformations such as partial unwrapping and tilting methods are introduced in order to change the locations of energetic frequencies when the current energetic frequencies are hard or impossible to find through the parallel projections directly. Through those manipulations, each frequency vector $\vect{w}_j$ is recovered in an entry-wise fashion. In this way, the linear dependence of runtime and sampling complexities on the dimension $d$ could be shown empirically, which is a great improvement when compared to the $d$-dimensional FFT with the exponential dependence on $d$. The transformations and projections occur in
the physical domain, which provides the separation of the frequency vectors in the Fourier domain. That is, each $\vect{w}_j$ is transformed to $\vect{w}_j'$ and then projected onto several lines, and these can be done by manipulating the sampling points in the physical domain $D$. Let $g:D' \rightarrow D$ represent those transformations where $D'$ is $d$ or less dimensional space and is determined by each transformation. We assume that $D'$ has $d'$ dimensions. A new function $h:D'\rightarrow \mathbb{C}$ is defined as a composition of $f$ and $g$, i.e.,   
\begin{equation*}
h(\vect{t}):=f(g(\vect{t})).
\end{equation*}
We note that $h$ is still $s$-sparse in the $d'$-dimensional Fourier domain, $[-N'/2, N'/2)^{d'}\cap \mathbb{Z}^{d'}$, whose bandwidth $N'$ depends on each transformation. 
For example, consider a 4-dimensional function $f$ with the Fourier domain, $[-N/2, N/2)^4\cap \mathbb{Z}^4$. If a partial unwrapping is applied to $f$ which unwraps each $\vect{w}_j=(w_{j,1}, w_{j,2}, w_{j,3}, w_{j,4})$ to $\vect{w}_j'=(w_{j,1}+N w_{j,2}, w_{j,3}+N w_{j,4})$ then it implies that $g:(t_1,t_2)\rightarrow (t_1,Nt_1,t_2,Nt_2)$ and $h$ has the Fourier domain, $[-N^2/2, N^2/2)^2\cap \mathbb{Z}^2$ where accordingly $N'=N^2$ and $d'=2$. This is one example and there are variations of partial unwrapping methods and tilting methods which can be found in \cite{2016arXiv160607407C}.
Using samples of $h$(or $f$), the transformed frequency vectors $\vect{w}_j'$ and corresponding Fourier coefficients $a_j$ are found through the parallel projection method and $\vect{w}_j'$ are transformed back to $\vect{w}_j$. Now, we introduce how $\vect{w}_j'=(w_{j,1}, w_{j,2}, \cdots, w_{j,d'})$ can be recovered element-wisely using the parallel projection method and the ideas from one-dimensional sparse Fourier algorithm. Let $p$ be a prime number greater than a constant multiple of $s$, i.e., $p>cs$ for some constant $c$, $j$ be a fixed integer among $\{1,2,\cdots,s\}$, $k$ be fixed among the set $\{1,2,\cdots d'\}$ and $\vect{e}_k$ be a vector with all zero entries but $1$ at the index $k$. Furthermore, $\epsilon$ is defined as a positive number $\leq 1/N'$. To recover $k$-th element of each $\vect{w}_j'$, we use two sets of $p$-length equispaced samples $\mathbf{h}_{p}^{\widetilde{k}}$ and $\mathbf{h}_{p;\epsilon}^{\widetilde{k};k}$ as follows,
\begin{equation*}
\mathbf{h}_{p}^{\widetilde{k}}[\ell]:=h\left( \frac{\ell}{p} \vect{e}_{\widetilde{k}} \right) \quad \text{and} \quad
\mathbf{h}_{p;\epsilon}^{\widetilde{k};k}[\ell]:=h\left( \frac{\ell}{p} \vect{e}_{\widetilde{k}} + \epsilon \vect{e}_k\right)
\end{equation*} 
where $\ell=0,1,\cdots,p-1$ and $\widetilde{k}\in\{1,2, \cdots, d'\}$ is the index of coordinate axis where a particular $\vect{w}'_{\bar{j}}$ has $\widetilde{k}$-th element, $w'_{\bar{j},\widetilde{k}}$,  different from the $\widetilde{k}$-th elements of any other energetic frequency vectors. In this case, we refer to the above phenomenon as ``no collision from projection". At the same time, if there is no {\em collision modulo} $p$, i.e., $w'_{\bar{j},\widetilde{k}}$ has the unique remainder modulo $p$ from others then the discrete Fourier transform of each sample set is
\begin{align}
\mathcal{F}\left(\mathbf{h}_{p}^{\widetilde{k}}\right)[m]&=p \sum_{w'_{j,k}\equiv m \bmod {p}} a_j=pa_{\bar{j}} ,\qquad \text{and} \nonumber\\
\mathcal{F}\left(\mathbf{h}_{p;\epsilon}^{\widetilde{k};k}\right)[m]&=p \sum_{w'_{j,k}\equiv m \bmod {p}} a_j e^{2\pi i w'_{j,k} \epsilon}=pa_{\bar{j}} e^{2\pi i w'_{\bar{j},k} \epsilon},
\label{eqn:shift}
\end{align}
respectively. The equations above in (\ref{eqn:shift}) give a unique entry for the $k^{\rm th}$ element assuming there does not exist a collision modulo $p$ of the vetor projected onto the $k^{\rm th}$ axis. Hence, in the above equations, when the second equalities hold, we can recover  $a_{\bar{j}}$ and $w'_{\bar{j},k}$ as follows,
\begin{equation}
w'_{\bar{j},k}=\frac{1}{2 \pi \epsilon}{\rm Arg}\left( \frac{\mathcal{F}\left(\mathbf{h}_{p;\epsilon}^{\widetilde{k};k}\right)[m]}{\mathcal{F}\left(\mathbf{h}_{p}^{\widetilde{k}}\right)[m]} \right) \qquad \text{and} \qquad a_{\bar{j}}=\frac{\mathcal{F}\left(\mathbf{h}_{p}^{\widetilde{k}}\right)[m]}{p},
\label{eqn:noiselessFreqCoeff}
\end{equation}
where the function ${\rm Arg}(z)$ is defined to be the argument of $z$ in the branch $[-\pi,\pi)$. The right choice of the branch and the shift size $\epsilon\leq \frac{1}{N'}$ make it possible to find the correct $w'_{\bar{j},k}$. Algorithmically, the two kinds of collisions are guaranteed not too happen using the following test:
\begin{equation}
\left|\frac{ \mathcal{F} \left(\mathbf{h}_{p;\epsilon}^{\widetilde{k};k}\right)[m]}{\mathcal{F}\left(\mathbf{h}_{p}^{\widetilde{k}}\right)[m]} \right|=1,
\label{eqn:tests}
\end{equation}
which is inspired by the test used in \cite{lawlor2013adaptive}. Practically, we put some threshold $\tau>0$ so that if the difference between that the left and right-hand sides is less than $\tau$, then we conclude that there are no collisions of both kinds. In this case, each $w_{\bar{j},k}$ for $k=1,2,\cdots,d'$ can be recovered using a pair of sets $\mathbf{h}_{p}^{\widetilde{k}}$ and $\mathbf{h}_{p;\epsilon}^{\widetilde{k};k}$, respectively. Otherwise, we take another prime number for sample length, switch the index of coordinate axis for the projection, update the samples by eliminating the influence from previously found fourier modes and repeat our procedure as before. Switching the coordinate axis and updating samples reduce the occasions of {\em collision from projection}. In \cite{lawlor2013adaptive}, moreover, it is proved that the probability is very low when $d'$ is large that all remaining frequency vectors have {\em collisions from projection} onto all coordinate axes, which we call the ``{\em worst case scenario}''. In the ``{\em worst case scenario}'', the parallel projections we just used do not work and thus, we need a rotation mapping, $g'$, defining another function $h'=f \circ g'$.

\subsection{Noise Model} \label{noise}

In this section, we introduce a model system which contains noise. We use this model system to quantify the behavior of the algorithm in the presence of noise.
The algorithm we introduced in Section \ref{sec:notation} works well when $h$ (or $f$) is exactly $s$-sparse and the samples from $h$ (or $f$) are not noisy. The high-dimenional sparse Fourier transform, described in the previous section, is not robust to noise since in order to find entries of energetic frequency vectors we compute the fraction ${ \mathcal{F} \left(\mathbf{h}_{p;\epsilon}^{\widetilde{k};k}\right)[m]}\big/{\mathcal{F}\left(\mathbf{h}_{p}^{\widetilde{k}}\right)[m]}$ which is sensitive to noise. The model we consider here is, 
\begin{equation*}
\mathbf{r}_{p}^{\widetilde{k}}[\ell]:=\mathbf{h}_{p}^{\widetilde{k}}[\ell]+n_{\ell}=h\left(\frac{\ell}{p}\vect{e}_{\widetilde{k}}\right)+n_{\ell}
\end{equation*}
where $h$ is the tranformed $f$ defined in the previous section,  $\ell=0,1,\cdots, p-1$, and $\vect{n}:=(n_0, n_1, \cdots, n_{p-1})$ is a complex Gaussian random variable with mean $\vect{0}$ and variance $\sigma^2I$. 
If we apply DFT to this sample set, we get
\begin{equation}
\mathcal{F}\left(\mathbf{r}_{p}^{\widetilde{k}}\right)[m]=\mathcal{F} \left( \mathbf{h}_{p}^{\widetilde{k}}\right) [m]+\sum_{\ell=0}^{p-1} n_{\ell}e^{-2\pi i m \frac{\ell}{p} }
\label{eqn:dftUnshifted}
\end{equation}
Since $n_{\ell}$ are i.i.d Gaussian variables, the expectation and variance of the second term in (\ref{eqn:dftUnshifted}) are
\begin{equation*}
\mathbb{E}\left[ \sum_{\ell=0}^{p-1} n_{\ell}e^{-2\pi i m \frac{\ell}{p} } \right] = 0 
\end{equation*}
and
 \begin{equation*}
\text{Var}\left[ \sum_{\ell=0}^{p-1} n_{\ell}e^{-2\pi i m \frac{\ell}{p} } \right] = p\sigma^2,
\end{equation*}
respectively. Accordingly,
\begin{equation*}
\mathbb{E}\left[ \mathcal{F}\left(\mathbf{r}_{p}^{\widetilde{k}}\right)[m] \right] =  \mathcal{F}\left(\mathbf{h}_{p}^{\widetilde{k}}\right)[m]  
\end{equation*}
and
\begin{equation*}
\text{Var}\left[\mathcal{F}\left(\mathbf{r}_{p}^{\widetilde{k}}\right)[m] \right] = p\sigma^2.
\end{equation*}
For noisy shifted sample set $\mathbf{r}_{p;\epsilon}^{\widetilde{k};k}:=\mathbf{h}_{p;\epsilon}^{\widetilde{k};k}+\widetilde{\vect{n}}$ with an i.i.d Gaussian random vector $\widetilde{\vect{n}}$, we have likewise
\begin{equation*}
\mathbb{E}\left[ \mathcal{F}\left(\mathbf{r}_{p;\epsilon}^{\widetilde{k};k}\right)[m] \right] =  \mathcal{F}\left(\mathbf{h}_{p;\epsilon}^{\widetilde{k};k}\right)[m]  
\end{equation*}
and
\begin{equation*}
\text{Var}\left[\mathcal{F}\left(\mathbf{r}_{p;\epsilon}^{\widetilde{k};k}\right)[m] \right] = p\sigma^2.
\end{equation*}
In the case of $w'_{\bar{j},\widetilde{k}}$ not having collisions both from projection and modulo $p$, and $w'_{\bar{j},\widetilde{k}}\equiv m ~(\bmod ~p)$,
\begin{equation}  
\mathcal{F}\left(\mathbf{r}_{p}^{\widetilde{k}}\right)[m]=pa_{\bar{j}}+ \mathcal{O}(\sigma\sqrt{p})
\label{eqn:expectUnshift}
\end{equation}
and
 \begin{equation*} 
 \mathcal{F}\left(\mathbf{r}_{p;\epsilon}^{\widetilde{k};k}\right)[m]=pa_{\bar{j}}e^{2\pi i w'_{\bar{j},k} \epsilon}+ \mathcal{O}(\sigma\sqrt{p}) 
\end{equation*}
for each $k=1,2,\cdots,d'$. As a result, we get
\begin{equation}
\frac{\mathcal{F}\left(\mathbf{r}_{p;\epsilon}^{\widetilde{k};k}\right)[m]}{\mathcal{F}\left(\mathbf{r}_{p}^{\widetilde{k}}\right)[m]}
=e^{2\pi i w'_{\bar{j},k} \epsilon} + \mathcal{O} \left(\frac{\sigma}{a_{\bar{j}}\sqrt{p}}\right)
\label{eqn:noisy}
\end{equation}
and note that if there were no noise in samples, we only have the first term on the right side of (\ref{eqn:noisy}) which makes it possible to recover $w'_{\bar{j},k}$ by taking its argument and dividing it by $2\pi \epsilon$ as (\ref{eqn:noiselessFreqCoeff}). With noisy samples, however, it is corrupted with noise which is a multiple of $\frac{\sigma}{a_{\bar{j}}\sqrt{p}}$. Defining
\begin{equation}
\hat{w}_{\bar{j},k} :=\frac{1}{2\pi \epsilon} {\rm Arg}\left( \frac{\mathcal{F}\left(\mathbf{r}_{p;\epsilon}^{\widetilde{k};k}\right)[m]}{\mathcal{F}\left(\mathbf{r}_{p}^{\widetilde{k}}\right)[m]} \right),
\label{def:noisyFreq}
\end{equation}
we want to see how far $\hat{w}_{\bar{j},k}$ is from $w'_{\bar{j},k}$. For this purpose, we introduce the {\em Lee norm} associated with a lattice $\mathcal{L}$ in $\mathbb{R}$ as $\|z\|_{\mathcal{L}}:=\min_{y\in \mathcal{L}}|z-y|$ for $z\in \mathbb{R}$ and the related property that under the Lee norm associated with the lattice $2\pi \mathbb{Z}$,
\begin{equation}
\|{\rm Arg}(\gamma + \nu) - {\rm Arg}(\gamma) \|_{2\pi \mathbb{Z}} = \left\|{\rm Arg} \left(1 + \frac{\nu}{\gamma}\right)  \right\|_{2\pi \mathbb{Z}} \leq \frac{\pi}{2} \left|\frac{\nu}{\gamma} \right|,
\label{eqn:leenorm}
\end{equation}
where $|\gamma|\geq|\nu|$ with $\gamma, \nu \in \mathbb{C}$.
By choosing the sample length $p$ large enough depending on the least magnitude nonzero $a_{min}$ and the noise level $\sigma$, (\ref{eqn:leenorm}) can be applied to (\ref{eqn:noisy}) as follows,
\begin{equation*}
\left \|{\rm Arg} \left(\frac{\mathcal{F}\left(\mathbf{r}_{p;\epsilon}^{\widetilde{k};k}\right)[m]}{\mathcal{F}\left(\mathbf{r}_{p}^{\widetilde{k}}\right)[m]}\right) - 2\pi w'_{\bar{j},k} \epsilon \right\|_{2\pi \mathbb{Z}} \leq \mathcal{O}\left(\frac{\sigma}{|a_{\min}|\sqrt{p}}\right).
\end{equation*}
Consequently, 
\begin{equation}
\left\|\hat{w}_{\bar{j},k}  - w'_{\bar{j},k}\right\|_{ \mathbb{Z}} \leq \mathcal{O}\left(\frac{\sigma}{2\pi \epsilon~ |a_{\min}| \sqrt{p}}\right),
\label{eqn:dist}
\end{equation}
which implies that the error of our estimate $\hat{w}_{\bar{j},k}$ to $w'_{\bar{j},k}$ is controlled by the size of $\frac{\sigma}{\epsilon~ |a_{\min}| \sqrt{p}}$. Thus, $p$ needs to be chosen carefully depending on $\frac{\sigma}{\epsilon~ |a_{\min}|}$. On the other hand, from (\ref{eqn:expectUnshift}), we can approximate the corresponding coefficient $a_{\bar{j}}$ with the error of size $\mathcal{O}(\sigma/\sqrt{p})$ as follows
\begin{equation}
a_{\bar{j}} = \frac{1}{p}\mathcal{F}\left(\mathbf{r}_{p}^{\widetilde{k}}\right)[m] + \mathcal{O}\left(\frac{\sigma}{\sqrt{p}}\right).
\label{eqn:coeffEst}
\end{equation}

\section{Multiscale Method} \label{sec:multiscale}
\setcounter{equation}{0}

Here we introduce the multiscale approach for recovering frequencies in the noisy setting. The method was introduced in \cite{christlieb2016multiscale}. The method is part of the overall sublinear algorithm described in Section \ref{sec:alg}. The basic idea of the algorithm is to find the most significant bits which are the least susceptible to noise. Then the algorithm subtracts the leading bits and shifts the remaining bits to the most significant digits. As we describe here in Section \ref{sec:multiscale}, this will decrease the impact of the noise by the use of larger shifting size $\epsilon_{\alpha}$ recovering the next most significant bits. This is repeated to recover the entire entries of the frequency vector.

In \cite{christlieb2016multiscale}, {\em rounding} and {\em multiscale} methods for the one-dimensional sparse Fourier algorithm for noisy data were introduced. Both methods use the fact that the peaks of DFT are robust to relatively high noise, i.e., $\ell$ can be correctly found such that $w \equiv \ell (\bmod ~p)$ for an energetic frequency $w$. The {\em rounding} method is efficient when $\sigma$ is relatively small, which approximates such $w:=bp+\ell$ for some $b$ up to $p/2$ error and rounds a  multiple of $p$ in order to get the correct $b$. It was shown that $p\geq \max\{c_1s,c_2(\frac{\sigma}{\epsilon a_{\min}})^{2/3}\}$ for some constants $c_1$ and $c_2$ makes it possible to correctly find $w$ through the {\em rounding} method. On the other hand, the {\em multiscale} method was introduced for relatively large $\sigma$. It prevents $p$ from becoming too large, which happens for large $\sigma$ in {\em the rounding method}. In this section, we focus on extending the {\em multiscale} method to recover high dimensional frequencies by gradually fixing each entry estimation with several shifts $\epsilon_{\alpha}$. 

\subsection{Description of Frequency Entry Estimation}

In this section, we give an overview of how {\em the multiscale method} works.
Let ${j}\in \{1,2,\cdots, s\}$ and  $k \in \{1,2,\cdots, d'\}$ be fixed. The target frequency entry $w'_{j,k}$ is assumed not to have collisions both from projection and modulo $p$.
We start with a coarse estimation $w^0_{j,k}$ of $w'_{j,k}$ defined by
\begin{equation*}
w^0_{j,k}:= \frac{1}{2\pi \epsilon_0} {\rm Arg} \left( \frac{\mathcal{F}\left(\mathbf{r}_{p;\epsilon_0}^{\widetilde{k};k}\right)[m]}{\mathcal{F}\left(\mathbf{r}_{p}^{\widetilde{k}}\right)[m]} \right),
\end{equation*}
where $\epsilon_0 \leq 1/N'$. Then $w^0_{j,k} \equiv w'_{j,k} (\bmod ~p)$ even though it is not guaranteed that $w^0_{j,k}=w'_{j,k}$. Thus, we need to improve the approximation. With each correction, the solution is improved by $l$ digits where $l$ depends on the parameters that are chosen in the method as well as noise. Each correction term is calculated with a choice of growing $\epsilon_{\alpha}>1/N'$, i.e., $\epsilon_{\alpha-1}<\epsilon_{\alpha}$ for all $\alpha\geq 1$. With the initialization $b_{k,0}:=\epsilon_0 w^0_{j,k}$, the correction terms are calculated in the following way,
\begin{equation*}
b_{k,\alpha}:= \frac{1}{2\pi}{\rm Arg}\left( \frac{\mathcal{F}\left(\mathbf{r}_{p;\epsilon_{\alpha}}^{\widetilde{k};k}\right)[m]}{\mathcal{F}\left(\mathbf{r}_{p}^{\widetilde{k}}\right)[m]} \right)  
\end{equation*}
and
\begin{equation*}
w^{\alpha}_{j,k}:=w^{\alpha-1}_{j,k}+\frac{\left( b_{k,\alpha}-\epsilon_{\alpha} w^{\alpha-1}_{j,k} \right)\left(\bmod ~\left[-\frac{1}{2},\frac{1}{2}~\right)\right)}{\epsilon_{\alpha}}
\end{equation*}
for $\alpha \geq 1$, where $x~(\bmod~ [-1/2,1/2))$ is defined to be the value $a\in[-1/2, 1/2)$ such that $x\equiv a~(\bmod~ 1)$.
Using the fact,
\begin{equation*}
b_{k,\alpha} \approx \epsilon_{\alpha} w'_{j,k} \left(\bmod~\left[-\frac{1}{2}, \frac{1}{2} \right)\right),
\end{equation*}
the error $w'_{j,k}-w^{\alpha-1}_{j,k}$ can be estimated as
\begin{align*}
\epsilon_{\alpha} (w'_{j,k}- w^{\alpha-1}_{j,k}) &= \epsilon_{\alpha} w'_{j,k} -\epsilon_{\alpha}w^{\alpha-1}_{j,k}\\
&\approx (b_{k,\alpha} -\epsilon_{\alpha} w^{\alpha-1}_{j,k}) \left(\bmod~\left[-\frac{1}{2}, \frac{1}{2} \right)\right)
\end{align*}
and thus, in a similar manner to (\ref{eqn:dist}),
\begin{align}
\mathcal{O}\left(\frac{\sigma}{\epsilon_{\alpha}a_{\min}\sqrt{p}}\right)&=( w'_{j,k}-w^{\alpha-1}_{j,k}) - \frac{\left( b_{k,\alpha}-\epsilon_{\alpha} w^{\alpha-1}_{j,k} \right)\left(\bmod ~\left[-\frac{1}{2},\frac{1}{2}~\right)\right)}{\epsilon_{\alpha}}\\
&= w'_{j,k} -\left(w^{\alpha-1}_{j,k} + \frac{\left( b_{k,\alpha}-\epsilon_{\alpha} w^{\alpha-1}_{j,k} \right)\left(\bmod ~\left[-\frac{1}{2},\frac{1}{2}~\right)\right)}{\epsilon_{\alpha}}\right)\\
&=w'_{j,k}-w^{\alpha}_{j,k}.
\label{eqn:errorEst}
\end{align}
As the correction is repeated with larger $\epsilon_{\alpha}$, the error of the estimate decreases. In other words, we approximate $w'_{j,k}$ by its most significant bits and the next significant bits repeatedly. The performance of this multiscale method is shown in detail in the next section. 

\subsection{Analysis of Multiscale Method}
 
 In this section, we establish the multiscale algorithm recovering a fixed number of additional bits of the frequency with each iteration, and we further establish the rate of reconstruction.
 The following theorem shows how the correction term $c_{k,\alpha}/\epsilon_{\alpha}$ is constructed in each iteration and how large the error of estimate $w^{\alpha}_{j,k}$ after $M$ iterations is in the multiscale frequency entry estimation procedure.

\begin{mytheorem}
Let $j \in \{1,2,\cdots,s\} , k \in \{1,2,\cdots,d'\}$ be fixed and $w'_{j,k} \in \left[ -\frac{N'}{2}, \frac{N'}{2} \right)$. Let $0<\epsilon_0<\epsilon_1<\cdots<\epsilon_M$ and $b_{k,0}, b_{k,1}, \cdots, b_{k,M} \in \mathbb{R}$ such that 
\begin{equation*}
\| \epsilon_{\alpha}w'_{j,k} -b_{k,\alpha} \|_{\mathbb{Z}}<\delta, \qquad 0\leq\alpha\leq M
\end{equation*}
where $0<\delta<\frac{1}{4}$. Assume that $\epsilon_0\leq \frac{1-2\delta}{N'}$ and $\beta_{\alpha}:=\frac{\epsilon_{\alpha}}{\epsilon_{\alpha-1}}\leq \frac{1-2\delta}{2\delta}$. Then there exist $c_{k,0}, c_{k,1}, \cdots, c_{k,M}\in \mathbb{R}$, each computable from $\{\epsilon_{\alpha}\}$ and $\{b_{k,\alpha}\}$ such that
\begin{equation*}
\left| \widetilde{w}_{j,k}-w'_{j,k} \right| \leq \frac{\delta}{\epsilon_0}\prod_{\alpha=1}^{M} \beta_{\alpha}^{-1}, \qquad
where \qquad \widetilde{w}_{j,k} := \sum_{\alpha=0}^{M} \frac{c_{k,\alpha}}{\epsilon_{\alpha}}.
\end{equation*}
\label{thm:freqEst}
\end{mytheorem}

\begin{proof}
The proof is the same as the proof of Theorem 4.2 in \cite{christlieb2016multiscale} since each entry of frequency vectors is corrected in the same way as each one-dimensional frequency $w$ in \cite{christlieb2016multiscale} is. Each $c_{k,\alpha}$ is defined as
\begin{align*}
c_{k,0}&:=b_{k,0}\\
c_{k,\alpha}&:=b_{k,\alpha}-\epsilon_{\alpha}\lambda_{k,\alpha-1}\left(\bmod \left[-\frac{1}{2}, \frac{1}{2}\right)\right),
\end{align*}
for $\alpha \geq 1$ where $\lambda_{k,\alpha}=c_{k,\alpha}/\epsilon_{\alpha}$ for $\alpha \geq 0$.
\end{proof}

\begin{corollary}
Assume that we let $\beta_{\alpha}=\beta$ in Theorem \ref{thm:freqEst} where $\beta\leq (1-2\delta)/(2\delta)$, i.e., $\epsilon_{\alpha}=\beta^{\alpha}\epsilon_0$ for all $\alpha \geq 1$. Let $p>0$ and $M \geq \left\lfloor \log_{\beta}\frac{2\delta}{\epsilon_0} \right\rfloor+1$. Then,
\begin{equation*}
\left| \widetilde{w}_{j,k}-w'_{j,k} \right| \leq \frac{\delta}{\epsilon_0}\beta^{-M}<\frac{1}{2}.
\end{equation*}
\label{cor:iterNum}
\end{corollary}

\begin{proof}
This is straightforward corollary of Theorem \ref{thm:freqEst}.
\end{proof}

Corollary \ref{cor:iterNum} looks very similar to Corollary 4.3 in \cite{christlieb2016multiscale}. However, it is different in that the iteration number is increased in order to make the error between $\widetilde{w}_{j,k}$ and $w'_{j,k}$ is less than $1/2$ instead of $p/2$. In \cite{christlieb2016multiscale}, the remainder $m$ modulo $p$ of each energetic $w$ is known by sorting out the $s$ largest DFT components of $p$-length unshifted sample vector so that $p/2$ error bound is enough to guarantee the exact recovery of $w$. On the other hand, in the high-dimensional setting of this paper, the remainders of all entries of $w'_{j,k}$ are not known. Instead, the remainder of $w'_{j,\widetilde{k}}$ is only known where $\widetilde{k}$ is the index of coordinate axis where the frequencies are projected. Thus, we decrease the error further by enlarging the iteration number $M$ and are able to recover the  exact $w'_{j,k}$ by rounding when the error is less $1/2$.

\begin{remark}
Similar to Theorem 4.4 in \cite{christlieb2016multiscale}, the admissible size of $\delta$ can be estimated  as follows
\begin{equation}
\delta=\min\left( \frac{1-\epsilon_0 N'}{2}, \frac{1}{2\beta+2} \right),
\label{eqn:deltaSize}
\end{equation}
under the assumption $\epsilon_0\leq \frac{1-2\delta}{N'}$ of Corollary \ref{cor:iterNum} together with the assumption $\beta\leq \frac{1-2\delta}{2\delta}$ of Theorem \ref{thm:freqEst}.
\end{remark}

\section{Algorithm} \label{sec:alg}
\setcounter{equation}{0}

In this section, we present the overall multiscale high-dimensional sublinear sparse FFT which combines the multiscale method from Section \ref{sec:multiscale} with our previous work in \cite{2016arXiv160607407C}. In addition, we describe how key parameters of the algorithm are chosen. As will be discussed below, the choice of parameters are affected by the noise level, $\sigma$. It is important to know how the frequency entry estimation works in the algorithm and how the collision detection tests are modified  from the tests in \cite{2016arXiv160607407C} in order to make them tolerant of noise. Furthermore, we present the analysis of the average-case runtime and sampling complexity under assumption that the worst case scenario does not happen.

\subsection{Choice of $p$}

In this section, we establish the length of subsampling vector, $p$.
The sample length $p$ affects the total runtime complexity since the discrete Fourier transform is applied to all sample sets taken to recover the frequency entries. Due to this, we want to make it as small as possible. At the same time, however, we can see from  (\ref{eqn:dist}) that the error between the target entry and its approximation becomes smaller if a larger $p$ is taken. Thus, as discussed in Section \ref{sec:multiscale}, if $p$ is large enough, then the {\em rounding method} instead of {\em multiscale method} can recover the exact frequency by rounding (\ref{def:noisyFreq}) to the nearest integer of the form $pv+ m$ with an integer $v$. In this case, $p$ is large enough to diminish the influence of $\sigma$, and therefore if $\sigma$ is large, so is $p$. Instead, the {\em multiscale method} makes it possible to enlarge $p$ moderately. From Theorem \ref{thm:freqEst}, we get 
\begin{equation}
|w^{\alpha+1}_{j,k}-w'_{j,k}| < \frac{\delta}{\epsilon_{\alpha+1}}
\label{eqn:error1}
\end{equation} 
and (\ref{eqn:errorEst}) implies
\begin{equation}
|w^{\alpha+1}_{j,k}-w'_{j,k}| \leq \mathcal{O}\left(\frac{\sigma}{2 \pi\epsilon_{\alpha}~ |a_{\min}| \sqrt{p}}\right).
\label{eqn:error2}
\end{equation}
By putting the right side of (\ref{eqn:error2}) as $c_{\sigma} \frac{\sigma}{\epsilon_{\alpha}~ |a_{\min}| \sqrt{p}}$ with some constant $c_{\sigma}$ and equating both right sides of (\ref{eqn:error1}) and (\ref{eqn:error2}), $\beta:=\epsilon_{\alpha+1}/\epsilon_{\alpha}$ can be estimated as
\begin{equation}
\beta= \frac{2\pi \delta \sqrt{p}}{a_{\min}c_{\sigma}\sigma}.
\label{eqn:beta}
\end{equation}
$\beta$ determines the choice of $\epsilon_{\alpha}$ for each iteration, i.e., $\epsilon_{\alpha}=\beta^{\alpha}\epsilon_0$ where $\epsilon_0$ is chosen to be less than $1/ N'$.
Combining (\ref{eqn:deltaSize}) and (\ref{eqn:beta}), the sample length $p$ can be calculated as
\begin{equation*}
p=\left( \frac{ \beta(\beta+1)a_{\min}c_{\sigma}\sigma}{\pi} \right)^2
\end{equation*}
when $\epsilon_0=\frac{1}{2N'}$ and $\beta>1$, which implies $\delta=\frac{1}{2\beta+2}$. Eventually, the sample length $p$ for the multiscale method needs to satisfy
\begin{equation}
p>\max\left\{c_1s, \left(\frac{\beta(\beta+1)a_{\min}c_{\sigma}\sigma}{\pi} \right)^2\right\},
\label{eqn:pEst}
\end{equation}
where $p>c_1s$ ensures the sample is long enough so that the $90\%$ of all energetic frequencies are not collided modulo $p$ on average which comes from the pigeonhole argument in \cite{lawlor2013adaptive}. 

\alglanguage{pseudocode}
\begin{algorithm}[H]
	\footnotesize
	\caption{Multiscale High-dimensional Sparse Fourier Algorithm Pseudo Code}
	\label{Algorithm1}
	\begin{algorithmic}[1]
		\Statex {\textbf{Input:}}{$f, g, s, N, d, N', d', \sigma, a_{\min}, c_{\sigma}, c_1, \eta, \beta$}
		\Statex {\textbf{Output:}}{$R$}
		\State $R \gets \emptyset$, $i \gets 0$
		\While { $|R|<s$}
		\State $s^{\ast} \gets s - |R|$
		\State $p \gets $ first prime number $\geq \max \left\{c_1 s^{\ast}, \left(\beta(\beta+1)a_{\min}c_{\sigma}\sigma/\pi\right)^2\right\}$
		\State $\tau \gets \frac{c_{\sigma}\sigma}{a_{\min}\sqrt{p}}$, $M \gets 1+ \lfloor \log_{\beta} N' \rfloor $
		\State $ \widetilde{k} \gets (i \bmod ~d')+1$
		\State $q\left({\mathbf{t}}\right) \gets \sum_{(\vect{w}',a_{\vect{w}'})\in R} a_{\vect{w}'} e^{2 \pi i \vect{w}' \cdot {\mathbf t}}$
		\For {$\ell = 0 \to p-1$} 
		\State ${r}^{\widetilde{k}}_{p}[\ell] \gets f\left(g\left(\frac{\ell}{p}{\vect{e}}_{\widetilde{k}} \right)\right)+n_{\ell} - q\left(\frac{\ell}{p}{\vect{e}}_{\widetilde{k}}\right)$ 
		\EndFor
		\State $\mathcal{F}\left(\vect{r}^{\widetilde{k}}_{p}\right) \gets { FFT}\left(\vect{r}^{\widetilde{k}}_{p}\right)$, $\mathcal{F}^{sort}\left(\vect{r}^{\widetilde{k}}_{p}\right) \gets SORT\left(\mathcal{F}\left(\vect{r}^{\widetilde{k}}_{p}\right)\right)$
		\State $vote \gets 0$
		\For {$\alpha = 0 \to M$}
		\State $\epsilon_{\alpha} \gets \frac{\beta^{\alpha}}{2N'}$
		\For {$k = 1 \to d'$}
		\For {$\ell = 0 \to p-1$}
		\State ${r}^{\widetilde{k};k}_{p;\epsilon_{\alpha}}[\ell] \gets f\left(g\left(\frac{\ell}{p}{\vect{e}}_{\widetilde{k}}+\epsilon_{\alpha}{\vect{e}}_{k} \right)\right)+n'_{\ell} - q\left(\frac{\ell}{p}{\vect{e}}_{\widetilde{k}}+\epsilon_{\alpha}{\vect{e}}_{k}\right)  $
		\EndFor
		\State $\mathcal{F}\left(\vect{r}^{\widetilde{k};k}_{p;\epsilon_{\alpha}}\right) \gets { FFT}\left(\vect{r}^{\widetilde{k};k}_{p;\epsilon_{\alpha}}\right)$, $\mathcal{F}^{sort}\left(\vect{r}^{\widetilde{k};k}_{p;\epsilon_{\alpha}}\right) \gets SORT\left(\mathcal{F}\left(\vect{r}^{\widetilde{k};k}_{p;\epsilon_{\alpha}}\right)\right)$
		\EndFor
		\For {$\ell =0 \to s^*-1$}
		
		\If {$\left|\frac{\left|\mathcal{F}^{sort}({r}^{\widetilde{k}}_{p})[\ell]\right|}{\left|\mathcal{F}^{sort}({r}^{\widetilde{k};k}_{p, \epsilon_{\alpha}})[\ell]\right|}-1 \right|>\tau$ for any $k=1,2,\cdots, d'$} {$vote \gets vote +1$} 
		\EndIf
		\For {$k = 1 \to d'$}
		\State $b_{k} \gets \frac{1}{2\pi}{\rm Arg}\left(\frac{\mathcal{F}^{sort}\left({r}^{\widetilde{k};k}_{p, \epsilon_{\alpha}}\right)[\ell]}{\mathcal{F}^{sort}\left({r}^{\widetilde{k}}_{p}\right)[\ell]}\right)$
		\If {$\alpha==0$}
		\State $w'_{k}\gets b_{k}/\epsilon_{\alpha}$
		\Else
		\State $w'_{k}\gets w'_{k}+\left(\left(b_{k}-\epsilon_{\alpha}w'_{k}\right)\left(\bmod~[-1/2,1/2)\right)\right)/\epsilon_{\alpha}$
		\EndIf 
		\If {$\alpha==M$}
		\State $w'_{k}\gets round(w'_{k})$
		\EndIf
		\EndFor
		\If {$vote\leq\eta(M+1)$}
		\State $a_{\vect{w}'}\gets\frac{1}{p}\mathcal{F}^{sort}\left(\vect{r}^{\widetilde{k}}_{p}\right)[\ell]$, $R \gets R\cup \left( {\vect{w}',a_{\vect{w}'}}\right)$
		\EndIf
		\EndFor
		\EndFor
		\State $i \gets i+1$
		\EndWhile
		\State {inverse-transform each ${{\vect{w}'}}$ in $d'$-D to ${\vect{w}}$ $d$-D} and restore it in $R$
	\end{algorithmic}
\end{algorithm}

\subsection{Collision Detection Tests}

As mentioned in Section \ref{sec:notation}, frequencies are recovered only when there are no collisions from the projection and the modulo $p$ division. These conditions are satisfied if 
\begin{equation}
\left|\left|\frac{ \mathcal{F} \left(\mathbf{r}_{p;\epsilon}^{\widetilde{k};k}\right)[m]}{\mathcal{F}\left(\mathbf{r}_{p}^{\widetilde{k}}\right)[m]} \right|-1\right|< \tau,
\label{eqn:pracTests}
\end{equation} 
for $k=1, \cdots, d'$ and some small $\tau>0$ which are the practical tests of (\ref{eqn:tests}). In our noisy setting, Equation (\ref{eqn:noisy}) implies that the left hand side of (\ref{eqn:pracTests}) is bounded above by $\mathcal{O}(\frac{\sigma}{a_{\min}\sqrt{p}})$. Thus, we set our threshold $\tau$ as a constant multiple of $\frac{\sigma}{a_{\min}\sqrt{p}}$. Moreover, since we iteratively update the estimates $w^{\alpha}_{j,k}$ for $\alpha=0,1,\cdots,M$, we reject the estimate after $M$ iterations if the tests fail for more than $\eta(M+1)$ times for each $k$-th entry with $k=1,2,\cdots, d'$ where $\eta$ is a fraction$<1$. Numerical experiments indicate $\eta=\frac{1}{4}$ is a good number.

\subsection{Number of Iterations}
\label{sec:numIter}

In this section, we give a specific choice for the number of iterations in our multiscale algorithms. From Corollary \ref{cor:iterNum}, $M \geq \left\lfloor \log_{\beta}\frac{2\delta}{\epsilon_0} \right\rfloor+1$ guarantees we get the approximation error, $\left| \widetilde{w}_{j,k}-w'_{j,k} \right| <\frac{1}{2}$ which is required to recover the exact $w'_{j,k}$ by rounding $\widetilde{w}_{j,k}$ to the nearest integer. With our choice of $\epsilon=\frac{1}{2N'}$ and the fact that $\delta<1$, $M=\left\lfloor \log_{\beta}  N' \right\rfloor+1$ suffices to satisfy the $1/2$ error bound. For example, if each $\vect{w}_j \in  \left[-\frac{N}{2}, \frac{N}{2} \right)^d \cap \mathbb{Z}^d$ is partially unwrapped to some value in $ \left[-\frac{N^{d_1}}{2}, \frac{N^{d_1}}{2} \right)^{d_2} \cap \mathbb{Z}^{d_2}$ where $d_1$ and $d_2$ are positive integers satisfying $d=d_1d_2$, then $d_2$ entries of each energetic frequency are recovered element-wisely after $M=\mathcal{O}(d_1\log N)$ iterations.

\subsection{Description of Our Pseudocode}

In this section, we explain the multiscale high-dimensional sparse Fourier transform whose pseudocode is provided in Algorithm \ref{Algorithm1}. The set $R$ contains the identified Fourier frequencies and their corresponding coefficients, and it is an empty set initially. Parameter $i$ is the counting number determining the index $\widetilde{k}$ of the coordinate axes where the frequencies are projected in line 6. Parameters $p, \tau$ and $M$ are determined as discussed in the previous sections. Function $q$ in line 7 is a function constructed from the previously found Fourier modes which is used in updating our samples in lines 9 and 17. In line 9, the unshifted samples $\vect{r}^{\widetilde{k}}_p$ corrupted by random Gaussian noise $n_{\ell}$ are taken and in line 11, DFT is applied to these samples and the transformed vector is sorted in the descending order of magnitude. Only its $s^{\ast}$ largest components are taken into account under the $s^{\ast}$-sparsity assumption. In the loop from line 12 through 39, the entries of frequencies corresponding to these $s^{\ast}$ components are estimated iteratively. The shift size $\epsilon_{\alpha}$ is updated in line 14 and we get the $d'$ number of length $p$ samples at the points shifted by $\epsilon_{\alpha}$ along each axis. Each length-$p$ sample will be used to approximate each entry. Similar to line 11, DFT is applied to each length-$p$ sample and the transformed  vector is sorted again following the index order of the sorted DFT of unshifted samples in line 19. In line 22, we check whether the $d'$ tests are passed at the same time or not. If not, the $vote$ is increased by 1.  From lines 24 through 34, the estimate $w_{j,k}'$ for $k$th entry is updated. Except when $\alpha=0$, $w_{j,k}'$ is improved by adding the correction term shown in line 29 in each iteration. In the last iteration when $\alpha=M$, $w_{j,k}'$ is rounded to the nearest integer in order to recover the exact entry, as guaranteed by Corollary \ref{cor:iterNum}. Whether this estimate is stored in the set $R$ or not is determined by checking if the $vote$ after $M$ iterations is less than $\eta(M+1)$. If it is less, this implies that the failure rate of the collision detection tests is less than $\eta$. Accordingly, we estimate $a_{\vect{w}'_j}$ from the DFT of unshifted samples and store $(\vect{w}'_j, a_{\vect{w}'_j})$ in $R$. The entire while loop repeats until $s$ energetic Fourier modes are all found switching the projection coordinate. Once we find all $s$ frequency vectors, each $\vect{w}'_j\in R$ is transformed to $d$-dimensional $\vect{w}_j=g^{-1}(\vect{w}'_j)$.

\subsection{Runtime and Sampling Complexity for the Average Case Signals Under No Worst-case Scenario Assumption}

In this section, we explain the performance of the multiscale high-dimensional sparse Fourier algorithms. In particular, we will restrict ourselves to the average-case analysis under the assumption that the wort-case scenario does not happen.

\begin{mytheorem}
	Let $f^z(\vect{x}) =f(\vect{x}) +z(\vect{x})$, where $\hat{f}(\vect{w})$ is $s$-sparse with each frequency $\vect{w}_j \in [- N/2, N/2)^d\cap \mathbb{Z}^d$ corresponding to the nonzero Fourier coefficient for $j\in \{ 1,2,\cdots,s \}$ and not forming any worst case scenario, and $z$ is complex i.i.d. Gaussian noise of variance $\sigma^2$. Moreover, suppose that $s>C(\beta(\beta+1)c_{\min}\sigma)^2$ for some constant $C$. Algorithm \ref{Algorithm1}, given $N, d, s, \beta$ with $N>5s$ and access to $f^z(\vect{x})$ returns a list of $s$ pairs $(\hat{\vect{w}}, a_{\hat{\vect{w}}})$ such that (i) each $\hat{\vect{w}}=\vect{w}_j$ for some $j\in \{ 1,2,\cdots,s \} $ and (ii) for each $\hat{\vect{w}}$,   $|a_{\vect{w}_j} - a_{\hat{\vect{w}}}| \leq C\sigma/ \sqrt{{s}}$.
	The average-case runtime and sampling complexity are
	\[
	\Theta(sd \log{s} \log{N})\qquad \text{ and } \qquad \Theta(s d\log{N}),
	\]
	respectively, over the class of random signals. 
\end{mytheorem}

\begin{proof}
	The difference of Algorithm \ref{Algorithm1} in this paper from Algorithm 2 in \cite{2016arXiv160607407C} appears in lines 13 through 39. Algorithm \ref{Algorithm1} has the multiscale frequency entry estimation. Thus, the average-case runtime and sampling complexity of Algorithm 2 from \cite{2016arXiv160607407C} is increased by a factor of $M$ which is the number of the repetition in the multiscale frequency entry estimation from Section \ref{sec:numIter}. Corollary \ref{cor:iterNum} ensures that the returned frequency vectors $\hat{\vect{w}}$ are correct, and the coefficient $a_{\hat{\vect{w}}}$ has the desired error bound from (\ref{eqn:coeffEst}).
\end{proof}

\section{Empirical Evaluation} \label{sec:eval}
\setcounter{equation}{0}

In this section, we show the empirical evaluation of the multiscale high-dimensional sparse Fourier algorithm. The empirical evaluation was done for test functions $f(\vect{x})$ which consist of $a_j$ randomly chosen from a unit circle in $\mathbb{C}$ and $\vect{w}_j$ randomly chosen from $[-N/2, N/2)^d \cap \mathbb{Z}^d$. Sparsity $s$ varied from $1$ to $2^{10}=1024$ by factor of 2.  The noise term added to each sample of $f$ came from the Gaussian distribution $\mathcal{N}(0, \sigma^2)$. Standard deviation $\sigma$ varied from $0.001$ to $0.512$ by factor of 2. The dimension $d$ was chosen to be $100$ and $1000$, and $N$ was chosen as $20$. The transformation $g$ was the one for partial unwrapping that was used in \cite{2016arXiv160607407C}, i.e., every $5$-dimensional subvector of each frequency vector was unwrapped to a one-dimensional vector and therefore each $100$ and $1000$-dimensional function $f$ was unwrapped to $20$ or $200$-dimensional function $f\circ g$ whose Fourier domain is  $[-20^5/2, 20^5/2)^{20} \cap \mathbb{Z}^{20}$ or  $[-20^5/2, 20^5/2)^{200} \cap \mathbb{Z}^{200}$, respectively. The input parameters $c_1=2$, $c_{\sigma}=6$, $\eta=1/4$ and $\beta=2.5$ were empirically chosen to balance the runtime and accuracy as in \cite{christlieb2016multiscale}. The initial shift size $\epsilon_0$ was set to $\frac{1}{2\cdot20^5}$.  All experiments are performed in MATLAB. 

The three plots in Figure \ref{fig:noise} show the average over 10 trials of the $\ell_1$ error, the number of samples, and the runtime in seconds as the noise level $\sigma$ changes. These values are in logarithm in the plots. Dimension $d$ and sparsity $s$ are fixed with $100$ and $256$, respectively. On the other hand, the other three plots in Figure \ref{fig:sparsity} show the average over 10 trials of the $\ell_1$ error, the number of samples, and the runtime in seconds as the sparsity $s$ changes when $d=100$ and $1000$. These values are in logarithm in the plots, and the noise level is fixed to $0.512$.

\begin{figure}
\centering

\begin{subfigure}[t]{0.9\textwidth}
\centering
  \includegraphics[width=0.90\textwidth, angle=0]{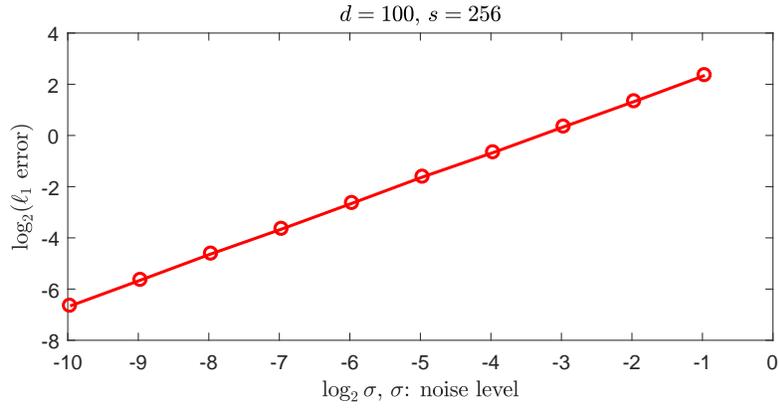}
\vspace{-0.15em}\caption{} \label{fig:error2}
\end{subfigure}

\begin{subfigure}[t]{0.9\textwidth}
\centering
  \includegraphics[width=0.90\textwidth, angle=0]{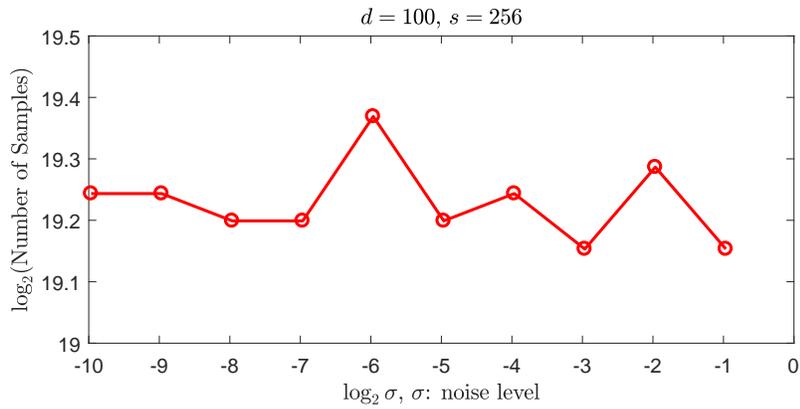}
\vspace{-0.15em}\caption{} \label{fig:sample2}
\end{subfigure}

\begin{subfigure}[t]{0.9\textwidth}
\centering
  \includegraphics[width=0.90\textwidth, angle=0]{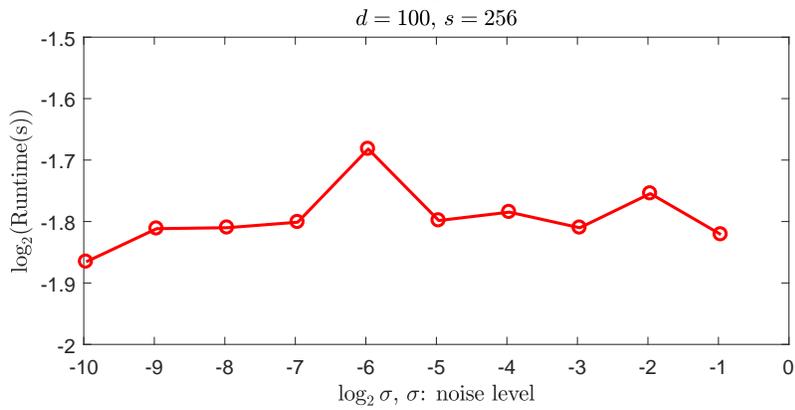}
\vspace{-0.15em}\caption{} \label{fig:time2}
 \end{subfigure}
 
 \caption{(a)Average $\ell_1$ error vs. noise level $\sigma$ in logarithm. (b)Average samples vs. noise level in logarithm. (c)Average runtime vs. noise level in logarithm.}
 \label{fig:noise}
\end{figure}

\begin{figure}
\centering

\begin{subfigure}[t]{0.9\textwidth}
\centering
  \includegraphics[width=0.90\textwidth, angle=0]{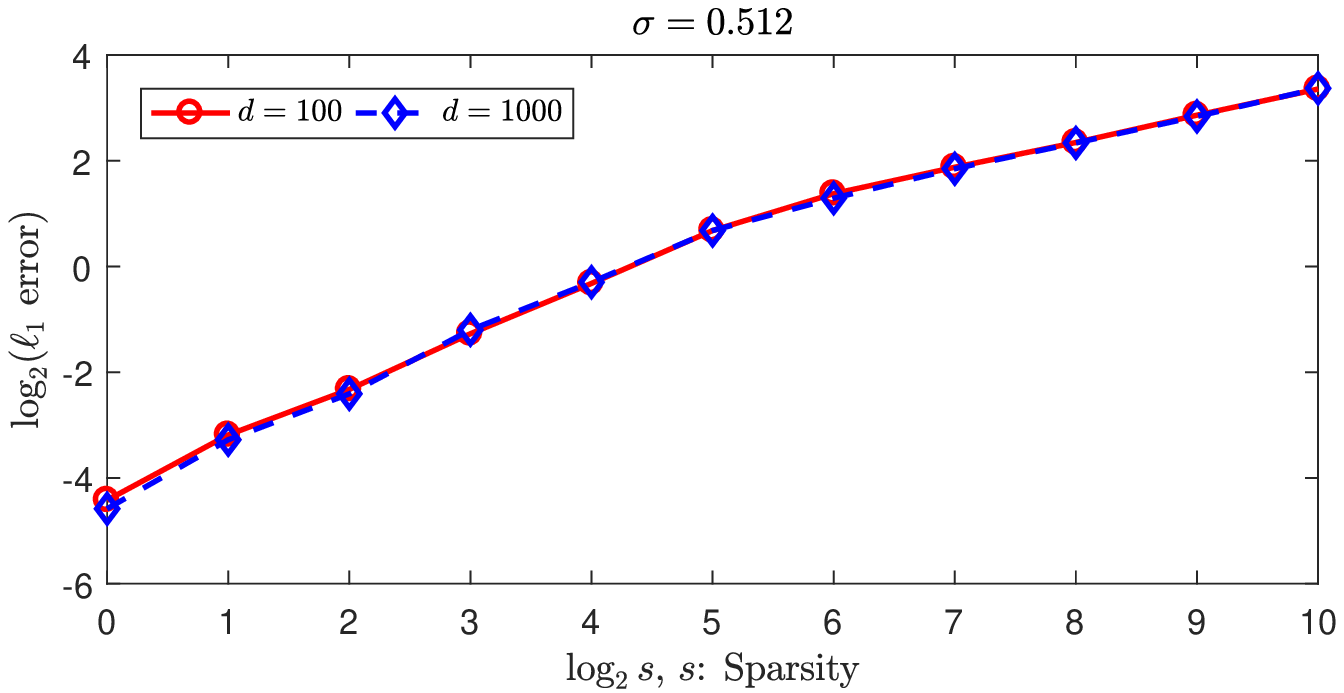}
\vspace{-0.15em}\caption{} \label{fig:error1}
\end{subfigure}

\begin{subfigure}[t]{0.9\textwidth}
\centering
  \includegraphics[width=0.90\textwidth, angle=0]{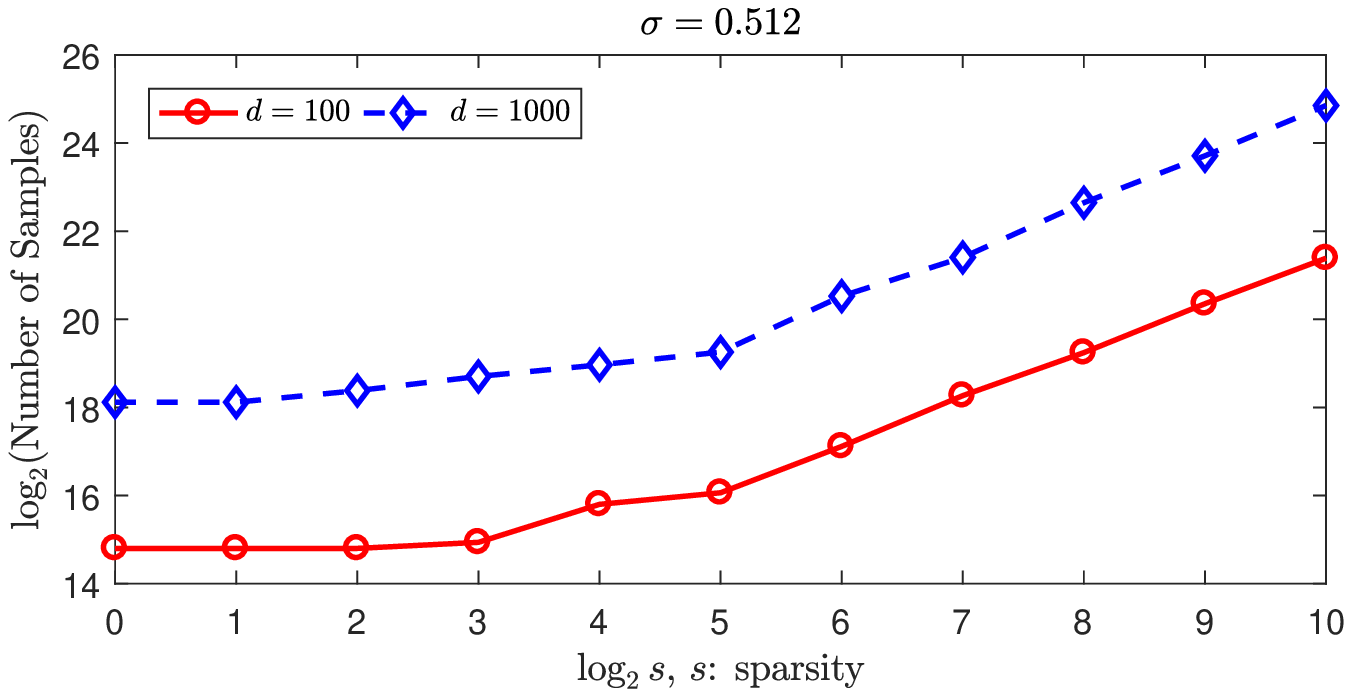}
\vspace{-0.15em}\caption{} \label{fig:sample1}
\end{subfigure}

\begin{subfigure}[t]{0.9\textwidth}
\centering
  \includegraphics[width=0.90\textwidth, angle=0]{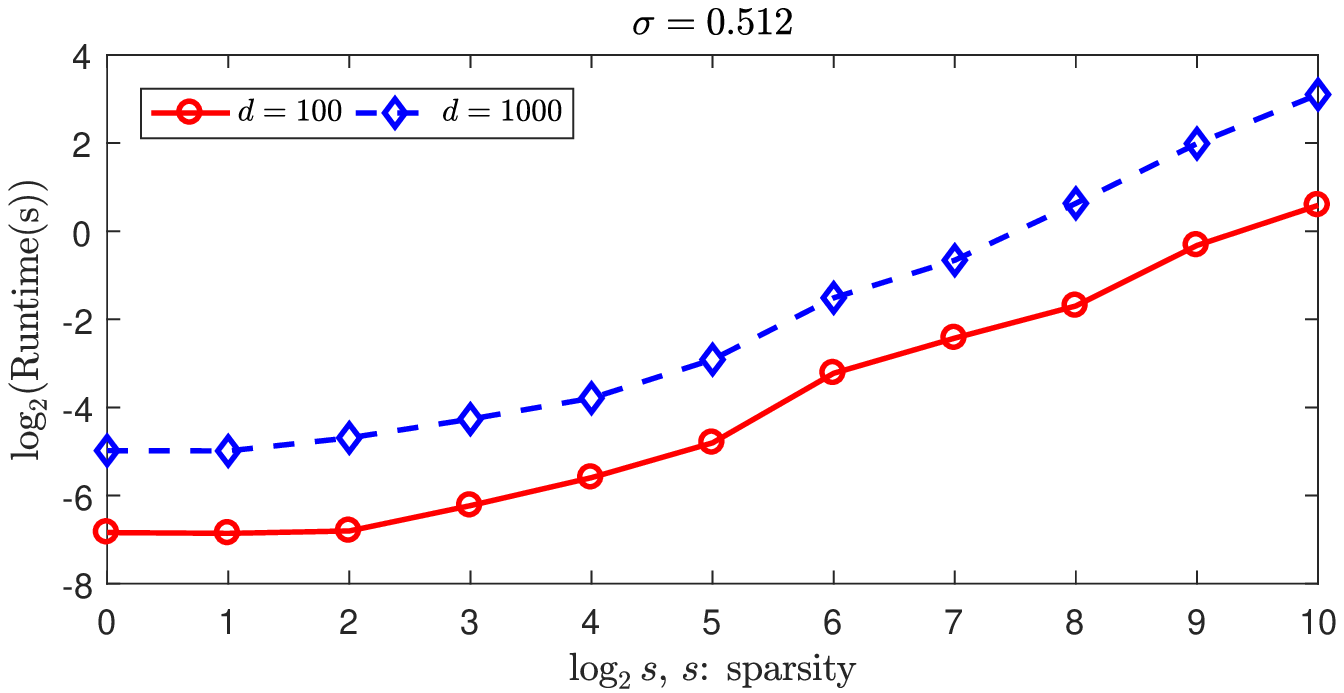}
\vspace{-0.15em}\caption{} \label{fig:time1}
 \end{subfigure}
 
 \caption{(a)Average $\ell_1$ error vs. sparsity $s$ in logarithm. (b)Average samples vs. sparsity $s$ in logarithm. (c)Average runtime vs. sparsity $s$ in logarithm.}
 \label{fig:sparsity}
\end{figure}

\subsection{Accuracy}

In this section, we do numerical experiments to investigate the accuracy of the algorithm.
In Figure \ref{fig:error2} and Figure \ref{fig:error1}, the $\ell_1$ errors of Fourier coefficient vectors are given under various parameter changes. Throughout all trials conducted in these experiments frequencies were always recovered exactly even for the noise level $\sigma=0.512$ which is relatively large compared to the true coefficients from the unit circle in $\mathbb{C}$. Thus, we can observe the errors only from coefficients whose size is a constant multiple of $\frac{\sigma}{\sqrt{p}}$ from (\ref{eqn:coeffEst}).  Due to the characteristic of the {\em multiscale method} which uses less samples compared to the {\em rounding method}, $\ell_1$ errors are relatively large in nature. From Figure \ref{fig:error2}, $\ell_1$ error looks increasingly linear as $\sigma$ increases, which meets our expectation. In Figure \ref{fig:error1}, the plot does not look exactly linear, but between $\log_2 s=5$ and $6$, there is a transition of slope. This is because the sample length $p$ from (\ref{eqn:pEst}) changes from $c_1s$ to $\left(\frac{\beta(\beta+1)a_{\min}c_{\sigma}\sigma}{\pi} \right)^2 $ during this transition.

\subsection{Sampling complexity}

In this section, we numerically explore the sampling complexity of the algorithm.
Sample numbers along $\sigma$ changes in Figure \ref{fig:sample2} seems irregular at first sight. Looking at the scale of vertical axis, however, we can see that the difference between maximum and minimum is less than $0.3$. Therefore, sampling complexity is not very affected by noise level. In Figure \ref{fig:sample1}, the red graph shows the average sample numbers as the sparsity increases when $d=100$ and the blue graph shows the ones when $d=1000$. Since our multiscale algorithm recovers each frequency entry iteratively using $\log N$ sets of $\mathcal{O}(s)$-length samples, the average-case sampling complexity is indeed $\mathcal{O}(sd\log N)$ when the worst-case scenario does not happen. Two graphs in Figure \ref{fig:sample1} look close to be linear excluding the transition between $\log_2 s=5$ and $6$, which again is caused by the change of $p$ from (\ref{eqn:pEst}). Moreover the difference between the values of the red and blue graphs are close to $3$, which implies that the sampling number depends linearly on $d$. The $d$-dimensional FFT whose sampling complexity is $\mathcal{O}(N^d)$ cannot deal with our high-dimensional problem computationally, whereas our algorithm uses only millions to billions of samples for reconstruction.

\subsection{Runtime complexity} 

In this section, we consider the average-case runtime complexity of our multiscale high-dimensional sparse Fourier transform.
Figures \ref{fig:time2} and \ref{fig:time1} demonstrate the average-case runtime complexity of the algorithm. The time
for evaluating the samples from functions is excluded when measuring the runtime. For the main algorithm, we demonstrated that it is $\mathcal{O}(sd\log s \log N )$ because for each entry recovery, DFT with $\mathcal{O}({s\log s})$ runtime complexity is applied in $d\log N$ iterations. 
In  Figures \ref{fig:time2}, the runtimes in seconds look irregular but the scale of vertical axis is less than $0.3$ so that we can conclude that  similar to sample numbers the runtime is not affected by $\sigma$ very much. Overall, it took less than a second on average. In Figure \ref{fig:time1}, the red graph represents the runtimes as the sparsity changes when $d=100$, and the blue graph represents the ones when $d=1000$. Those graphs do not look linear, but considering the average slope we can see that the runtime is increased by around $2^8$ while $s$ is increased by $2^{10}$. On the other hand, the difference between two graphs implies that the runtime complexity is linear in $d$. Compared to the FFT with runtime complexity of $\mathcal{O}(N^d \log N^d)$ which is impossible to be practical in high-dimensional problem,  our algorithm is quite effective, taking only a few seconds.

\section{Conclusion} \label{sec:con}
\setcounter{equation}{0}

In this paper, we developed a multiscale high-dimensional sparse Fourier algorithm recovering a few energetic Fourier modes using noisy samples. As the estimation error is controlled by the noise level $\sigma$ and the sample length $p$, larger $p$ reduces the error. Rather than recovering the frequencies in a single step, however, we choose multiscale approach in order to make the sample length $p$ increase moderately by improving the estimate iteratively through correction terms determined by a sequence of shifting sizes $\epsilon_{\alpha}$. We showed that a finite number of correction terms are enough to make the error smaller than $1/2$ so that we can reconstruct each integer frequency entry by rounding.  As a result, the algorithm has $\mathcal{O}(sd\log N)$ sampling complexity and $\mathcal{O}(sd\log s \log N)$ runtime complexity on average under the assumption that there is no worst case scenario happening by combining the result from \cite{2016arXiv160607407C} and \cite{christlieb2016multiscale}. 
In the numerical experiment we ran, with a noise of $\sigma=0.512$ we were able to recover 100\% of the frequencies up to dimension 1000. 
The methods introduced either in \cite{2016arXiv160607407C} and in this paper assume that we can get the measurement at any sample point. However, this is not always the case in practice. Our future work will be a modification of the algorithm to make it work for the given discrete signals using the idea of filtering from \cite{merhi2017new}. 

{\bf ACKNOWLEDGEMENTS}
We would like to thank Mark Iwen for his valuable advice. This research is supported in part by AFOSR grants FA9550-11-1-0281, FA9550-12-1-0343 and FA9550-12-1-0455, NSF grant DMS-1115709, and MSU Foundation grant SPG-RG100059, as well as Hong Kong Research Grant Council grants 16306415 and 16317416.

\bibliographystyle{abbrv}
\bibliography{multiscale_high_dimension}

\end{document}